\documentclass[11pt]{amsart}

\usepackage{epigamath}

\usepackage[english]{babel}

\numberwithin{equation}{section}

\usepackage[shortlabels]{enumitem}
\usepackage{amssymb}
\usepackage{amsmath}
\usepackage{amsthm}
\usepackage{caption}
\usepackage[curve,all]{xy}
\usepackage{array}
\usepackage{macro-longdash}
\usepackage{url}

\newtheorem{theorem}{Theorem}[section]
\newtheorem{lemma}[theorem]{Lemma}
\newtheorem{corollary}[theorem]{Corollary}
\newtheorem{proposition}[theorem]{Proposition}

\theoremstyle{definition}
\newtheorem{definition}[theorem]{Definition}

\theoremstyle{remark}
\newtheorem{remark}[theorem]{Remark}
\newtheorem{question}[theorem]{Question}
\newtheorem{example}[theorem]{Example}

\newcommand{\calD}{\mathcal{D}}

\newcommand{\frakf}{\mathfrak{f}}
\newcommand\frakS{\mathfrak{S}}

\newcommand{\bbT}{\mathbb{T}}
\newcommand{\bbF}{\mathbb{F}}
\newcommand{\bbP}{\mathbb{P}}
\newcommand{\bbG}{\mathbb{G}}
\newcommand{\la}{\langle}
\newcommand{\ra}{\rangle}

\DeclareMathOperator{\Sing}{Sing}
\DeclareMathOperator{\Ram}{Rm}
\DeclareMathOperator{\Bit}{Bit}
\DeclareMathOperator{\cha}{char}
\DeclareMathOperator{\Flex}{Flex}
\DeclareMathOperator{\rank}{rank}
\DeclareMathOperator{\Rey}{Rey}

\def\pr{{\mathrm{pr}}}
\def\half{\tfrac{1}{2}}
\def\da{\dasharrow}
\def\tsigma{\tilde{\sigma}}
\def\tX{\tilde{X}}

\newcommand{\lra}{\longrightarrow}

\setlist[enumerate,1]{label={\rm(\arabic*)}, ref={\rm\arabic*}}

\newcommand{\supth}[1]{\ensuremath{#1^{\mathrm{th}}}}

\EpigaVolumeYear{9}{2025} \EpigaArticleNr{23} \ReceivedOn{July 9, 2024}
\InFinalFormOn{April 22, 2025}
\AcceptedOn{May 17, 2025}

\title{Bitangent surfaces and involutions of quartic surfaces}
\titlemark{Bitangent surfaces and involutions}

\author{Igor Dolgachev}
\address{Department of Mathematics, University of Michigan, 525 East University Avenue, Ann Arbor, MI 48109-1109, USA}
\email{idolga@umich.edu}

\author{Shigeyuki Kond\=o}
\address{Graduate School of Mathematics, Nagoya University, Nagoya, 464-8602, Japan}
\email{kondo@math.nagoya-u.ac.jp}

\authormark{I.~Dolgachev and S.~Kond\=o}

\AbstractInEnglish{We study the congruence of bitangent lines of an irreducible surface in $\bbP^3$ in arbitrary characteristic, with special attention to quartic surfaces with rational double points and, in particular, Kummer quartic surfaces.}

\MSCclass{14J28}
\KeyWords{Congruence of lines, bitangent surface, Kummer quartic surface}

\acknowledgement{Research of the second author is partially supported by JSPS Grant-in-Aid for Scientific Research (A) No.~20H00112, (B) No.25K00906.}

\begin{document}



\maketitle

\begin{prelims}

\DisplayAbstractInEnglish

\bigskip

\DisplayKeyWords

\medskip

\DisplayMSCclass

\end{prelims}


\newpage

\setcounter{tocdepth}{1}

\tableofcontents


\section{Introduction} 
Let $G_1(\bbP^3)$ be the Grassmannian of lines in $\bbP^3$. For a point in $G_1(\bbP^3)$, the corresponding line in $\bbP^3$ is called a \emph{ray}.  A surface $F \subset G_1(\bbP^3)$ is called a \emph{congruence of lines}. Its algebraic class $[F]$ in the Chow ring $A^\bullet(G_1(\bbP^3))$ is determined by two numbers: the \emph{order} $m$ and the \emph{class} $n$.  The order (resp.\ class) is the number of lines on $F$ passing through a general point (resp.\ contained in a general plane) in $\bbP^3$. We will call the pair $(m,n)$ the \emph{bidegree} of $F$.  The number $m+n$ is equal to the degree of $F$ in the Pl\"ucker embedding $G_1(\bbP^3)\hookrightarrow \bbP^5$ (\textit{cf.} \cite[Section~X.2]{SR},
\cite{DoC,J}).

A surface $X$ in $\bbP^3$ defines a congruence of lines in $\bbP^3$ equal to the closure of the set of lines that are tangent to $X$ at two distinct points.  It is classically known as the \emph{bitangent surface}, and we denote it by $\Bit(X)$.  For a general $X$, its order and class are also classically known, and we will reproduce the calculation in this paper with the novelty of taking care of the case of characteristic $2$.

For example, it is known that, for a general surface of degree $d$, the bidegree of the surface $\Bit(X)$ is equal to $(\half d(d-1)(d-2)(d-3), \half d(d-2)(d^2-9))$; see \cite[Volume 1, Section~XI.279, p.~283]{Salmon}. In particular, the bidegree of the bitangent surface of a general quartic surface is equal to $(12,28).$ Moreover, the bitangent surface is a smooth surface of degree $40$ in $\bbP^5$. It is of general type with $p_g = 45$ and $K^2 =360$; see \cite{Tikhomirov, Welters}.
  
Although, for a general smooth quartic $X$, the surface $\Bit(X)$ is irreducible, for special quartic surfaces, even smooth ones, it can be reducible. It is known that, if a quartic surface $X$ is smooth and does not contain lines, then $\Bit(X)$ is smooth; see \cite{Zucconi1}. However, the surface $\Bit(X)$ could be reducible even when~$X$ is smooth but contains lines.  For example, a general Cayley's symmetroid quartic surface admits a smooth quartic model with reducible bitangent surface. One of its irreducible components is a Reye congruence of bidegree $(7,3)$; see \cite{Cossec} and \cite[Section~7.4]{Enriques2}.
 
Although admitting some mild singular points in $X$ does not change the bidegree, which may drastically change the surface $\Bit(X)$, it may even  become reducible. For example, when $X$ is realized as the focal surface of a congruence of lines of bidegree $(2,n)$ (the focal surface of a congruence $F$ of lines of order $n > 1$ is a surface in $\bbP^3$ such that all rays of $F$ are its bitangents; see, for example, \cite[Section~X.2]{SR}), the surface $\Bit(X)$ is always singular, although its singular points are ordinary double points. In the case $n = 2$, the congruence of lines is a del Pezzo surface of degree $4$, the focal surface is the famous $16$-nodal Kummer quartic surface $X$, see \cite{Kummer}, and the number of irreducible components of $\Bit(X)$ is equal to $22$; six of them are of bidegree $(2,2)$, and $16$ of them are of bidegree $(0,1)$. They are planes of lines contained in one of the $16$ planes intersecting $X$ along a double conic, a \emph{trope plane}.

The paper is addressing the problem of the decomposition of $\Bit(X)$ into irreducible components for quartic surfaces over an algebraically closed field of arbitrary characteristic $p$. The case $p = 2$ and $X$ being a Kummer quartic surface is of special importance to us.  Recall that smooth curves $C$ of genus $2$ in characteristic $2$ are divided into three types, that is, \emph{ordinary curves}, \emph{curves of\, $2$-rank $1$} and \emph{supersingular curves}.  The Jacobian $J(C)$ has four, two, or one 2-torsion point(s) accordingly.  The quotient surface $X=J(C)/(\iota)$ by the negation involution $\iota$ can be embedded in $\bbP^3$ with  image isomorphic to a quartic surface, the \emph{Kummer quartic surface} associated to $C$.  Instead of $16$ nodes in the case $p \ne 2$, the Kummer surface $X$ has four rational double points of type $D_4$, two rational double points of type $D_8$, or an elliptic singularity of type $\raise0.2ex\hbox{\textcircled{\scriptsize{4}}}_{0,1}^1$ in the sense of Wagreich \cite{W}.  The number of trope planes also drops; it is equal to $4, 2$, and $1$, respectively.
We will show that $\Bit(X)$ consists of the  irreducible components as in Table~\ref{table0}.

\begin{table}[ht]
 \centering
  \begin{tabular}{|c|c|c|}
   \hline  
  &Bidegree&Number\\ \hline 
  $p\ne 2$&$(2,2)$&6\\
 &$(0,1)$&16\\
\hline
 $p=2$, &$(1,1)$&3\\
ordinary &$(0,1)$&4\\ \hline
 $p=2$, &$(1,1)$&2\\
 $2$-rank 1&$(0,1)$&2\\ \hline
 $p=2$, &$(1,1)$&1\\
supersingular  &$(0,1)$&1\\  \hline
\end{tabular}
\caption{Irreducible components of $\mathrm{Bit}(X)$ for Kummer quartic surfaces $X$}\label{table0}
\end{table}
Observe that $(12,28)$ is a multiple of the sum of the bidegrees of irreducible components of the bitangent surfaces of ordinary Kummer surfaces in characteristic~$2$.  We do not know any explanation for this fact.

One reason for the dropping of the order of $\Bit(X)$ in characteristic $2$ is explained by the fact that the discriminant polynomial of a binary form is a square, and another is that irreducible components of bidegree $(1,1)$ appear instead of bidegree $(2,2)$.  On the other hand, the class, being equal to the number of bitangent lines to a general plane section $H$ of $X$, also drops. It is equal to $7, 4, 2$, or $1$ if the Hasse--Witt invariant of~$H$ is equal to $3, 2, 1$, or $0$, respectively; see \cite{R,SV}.

On the way, we discuss different kinds of involutions on quartic surfaces and their relationship to Cremona transformations. For example, an irreducible component of the bitangent surface $\Bit(X)$ with non-zero order and non-zero class defines a birational involution of $X$ with the pairs of tangency points of bitangent lines as its general orbits.

The plan of the paper is as follows.  In Section~\ref{sec1}, we will reproduce Salmon's proof for the formula of the bidegree of the bitangent surface of a general surface $X$ of degree $d$ in characteristic $0$ (Theorem~\ref{bidegrees}) and show it also holds  for smooth surfaces or surfaces with rational double points (Corollary~\ref{bidegrees2}).  In Section~\ref{sec2}, we discuss the classical results about the bitangent surface of a Kummer quartic surface in characteristic $p = 0$. They extend without change to the case $p\ne 2$.  In Section~\ref{involutions}, we discuss birational involutions of quartic surfaces.  In Section~\ref{sec3}, we give a proof (due to G.~Kemper) of the fact that the discriminant of a binary form in the case $p = 2$ is a square (Proposition~\ref{discri}).  This reduces the order of congruences of bitangent lines to half.  Finally, in  Sections~\ref{sec4}--\ref{sec6}, we determine the bidegrees of congruences of bitangent lines of Kummer quartic surfaces in characteristic $2$ according to ordinary, $2$-rank $1$, and supersingular, respectively (Theorems~\ref{(3,7)},~\ref{(2,4)}, and~\ref{(1,2)}).

Throughout the paper, we assume that the base field $\Bbbk$ is an algebraically closed field of characteristic $p\geq 0$.

\subsection*{Acknowledgments}
We would like to thank the referees for carefully reading the manuscript and for giving us many useful comments.

\section{Generalities on the surface of bitangent lines}\label{sec1}
Let $X$ be a normal surface of degree $d\ge 4$. We keep this assumption during the whole paper.  A line $\ell$ in~$\bbP^3$ is called a \emph{bitangent line} if either it is contained in $X$, or $X$ cuts out  a divisor $D = 2a+2b+D'$ in $\ell$.

Let $\Bit(X)\subset G_1(\bbP^3)$ be the variety of bitangent lines. If $\cha(\Bbbk)\ne 2$, we will show in this section that each irreducible component of $\Bit(X)$ is a surface. In Section~\ref{sec5}, we extend this result to the case where $\cha(\Bbbk) = 2$. Considered as a surface in $G_1(\bbP^3)$, $\Bit(X)$ has  bidegree $(m,n)$. Its degree in the Pl\"ucker embedding is equal to $m+n$.

Recall that $G_1(\bbP^3)$ contains two types of planes; one is called an {\it $\alpha$-plane} and consists of lines passing through a point of $\bbP^3$, and the other is called a {\it $\beta$-plane} and consists of lines contained in a plane of $\bbP^3$.

The following fact is, of course, well known, but to be sure that it is true without assumption on the characteristic, we supply a proof.

\begin{proposition}
An irreducible integral surface $S$ in $G_1(\bbP^3)$ of bidegree $(m,0)$ $($resp.\ $(0,n))$ is an $\alpha$-plane $($resp.\ $\beta$-plane$)$.  In particular, $m=1$ $($resp.\ $n=1)$.
\end{proposition}

\begin{proof}
Passing to the dual congruence of lines in the dual space of $\bbP^3$, it is enough to prove that any irreducible congruence of bidegree $(0,n)$ is a $\beta$-plane.  Let $Z_S = \{(x,\ell) \in \bbP^3\times S:x\in \ell\}$ be the restriction of the tautological projective bundle over $G_1(\bbP^3)$ to $S$. Since $S$ is irreducible, $Z_S$ is an irreducible $3$-fold. Since $m = 0$, its image $Y$ under the first projection is an irreducible subvariety of dimension $1$ or $2$.  Suppose $\dim Y = 1$. Then, the fibers of $Z_S\to Y$ are $\alpha$-planes that sweep at least one plane, and hence $S$ contains an $\alpha$-plane. So, $Y$ cannot be a curve.  Thus, $Y$ is a surface. The image of the fiber of $Z_S\to S$ over a point $y\in Y$ in $S$ is a curve contained in the $\alpha$-plane $\Omega(y)$ of lines through $y$.  The corresponding rays sweep a cone with vertex at $y$ contained in $Y$. Thus, $Y$ is a cone at each of its points. Since $Y$ is reduced, this can happen only if $Y$ is a plane. So, all rays of $S$ are contained in a plane, and hence, $S$ is a $\beta$-plane.
\end{proof}

\begin{proposition}\label{bidegrees}
Assume $\mathrm{char}(\Bbbk) = 0$, and let $X=V(F)$ be a normal surface of degree $d\geq 4$. Then, $\Bit(X)$ is a congruence of lines of bidegree $(m,n)$, with
$$
1\le m\le \half d(d-1)(d-2)(d-3), \quad n  =  \half d(d-2)(d^2-9).
$$
\end{proposition}

\begin{proof}
  We follow Salmon \cite[Volume 1, Section~XI.279, p.~283]{Salmon}.  Let $q$ be a general point in $\bbP^3$, and let $\ell = \la q,q'\ra$ be the line containing $q$ and tangent to $X$ at some point $q' = [x_0,y_0,z_0,w_0]$.  Without loss of generality, we may assume that $q = [0,0,0,1]$ and
$$
F = w^d+A_1w^{d-1}+\cdots+wA_{d-1}+A_d,
$$
where the $A_k$ are homogeneous forms of degree $k$ in $x,y,z$.

Plugging in the parametric equation $[s,t]\mapsto [sv+tv']$, where $[v] = q$ and $ [v'] = q'$, we get 
$$f:=F(sv+tv') = (s+tw_0)^d+\sum_{i=1}^dt^iA_i(x_0,y_0,z_0)(s+tw_0)^{d-i}.$$
By polarizing, we can rewrite it in the form
$$
f = \sum_{k+m=d}s^kt^{m}P_{v^k}(F)(v'),
$$
where $P_{v^k}(F)(v')$ is the value of the totally polarized symmetric multilinear form defined by $F$ at $(v,\ldots,v,v',\ldots,v')$. Geometrically, following the notation from \cite[Chapter 1]{DoC}, the locus of zeros of $P_{v^k}(F)(v')$ with fixed $v$ is the $\supth{k}$ polar $P_{q^k}(V(F))$ of the hypersurface $V(F)$.  Recall from \textit{loc.~cit.}~that the first \emph{polar hypersurface} $P_q(V(F))$ is the locus of zeros of $\sum a_i\frac{\partial F}{\partial x_i}$, where $q = [a_0,\ldots,a_n]$.  Since $q'\in X$, we get $P_{v^0}(F)(v') = F(v') = 0$.  Moreover, because $\ell$ is tangent to $X$ at $q'$, we obtain $P_{v}(F)(v') = 0$.  Thus, we can rewrite
$$
f = s^2g_{d-2}(s,t).
$$
The line $\ell$ is tangent to $X$ at some other point if and only if the binary form $g_{d-2}$ of degree $d-2$\
$$
g_{d-2}(s,t) = \sum_{k=0}^{d-2}s^kt^{d-2-k}P_{v^{k+2}}(F)(v')
$$
has a multiple root. 

Recall that the discriminant polynomial $D(a_0,\ldots,a_n)$ of a binary form $\sum_{i=0}^na_is^{n-i}t^i$ of degree $n$ is a homogeneous polynomial of degree $2(n-1)$.  The polynomial $D(a_0,\ldots,a_n)$ is also a bi-homogeneous polynomial in variables $a_0,\ldots,a_n$ of bidegree $(n(n-1),n(n-1))$ with respect to the action of $\bbG_m^2$ via
$$
(a_0,\ldots,a_n) \longmapsto (\lambda^na_0,\lambda^{n-1}\mu a_1,\ldots,\lambda\mu^{n-1}a_{n-1}, \mu^na_n).
$$
In other terms, it is a weighted homogeneous polynomial of degree $2(n-1)$ with the weights of the $a_k$ equal to $k$.
 
Applying this to the discriminant of the polynomial $g_{d-2}(s,t)$, we obtain that its discriminant $D(P_{v^d}(F)(v'),\ldots,P_{v^2}(F)(v'))$ is a polynomial of degree $\half (d-2)(d-3)$ in $x_0, y_0, z_0$. Hence we obtain that the locus of points $q'$ such that the line $\la q,q'\ra$ is tangent to $X$ at two points including $q'$ is contained in the intersection of hypersurfaces of degrees $d, d-1$ and $(d-2)(d-3)$.
This implies that the expected number of bitangent lines passing through $q$ is equal to $\half d(d-1)(d-2)(d-3)$.

The class $n$ of $\Bit(X)$ is equal to the number of bitangent lines to a general plane section $H$ of $X$. Since~$X$ is normal, the hyperplane section is a smooth plane curve of degree $d$.  Their number is well known classically; it is equal to $\half d(d-2)(d^2-9)$ (see \cite[Section~5.5.1, Formula (5.33)]{DoC}).

Since $\cha(\Bbbk) = 0$, the discriminants of the polynomials $g_{d-2}$ are not equal to zero.  This shows that $m\ge 1$.

To see that $\Bit(X)$ is a congruence of lines, \textit{i.e.}, each of its irreducible components is a surface, we consider the incidence variety
$$
M = \left\{(q,\ell)\in \bbP^3\times G_1\left(\bbP^3\right):\ell \in \Bit(X), q\in \ell\right\}.
$$
The fiber of the first projection consists of bitangent lines $\ell$ passing through $q$.  By the above, it consists of finitely many lines. This implies that each irreducible component of $M$ is of dimension $3$.  Since the fibers of the second projections are lines, the image of each irreducible component of $M$ is a surface.
\end{proof}

We denote by $\Flex(X)\subset G_1(\bbP^3)$ the variety of lines in $\bbP^3$ that either are contained in $X$ or intersect $X$ at some point with multiplicity at least $3$.  It intersects $\Bit(X)$ at the set of lines contained in $X$ and the set of bitangents that intersect $X$ at one point.

\begin{proposition}\label{flexes}
The expected bidegree of the congruence of lines $\Flex(X)$ is $(m,n)$, where $m\le d(d-1)(d-2)$ and $n \le 3d(d-2)$.
\end{proposition}

\begin{proof} 
Following the proof of the previous proposition, we find that the order $m$ of $\Flex(X)$ is equal to the degree of the reduced complete intersection $X\cap P_q(X)\cap P_{q^2}(X)$. We have $m\le d(d-1)(d-2)$.  The class~$n$ of $\Flex(X)$ is equal to the number of flex tangents in a general plane section $H$ of $X$.  It is equal to the number of intersection points of a plane section $H\cap X$ with its Hessian curve of degree $3(d-2)$. Clearly, $m\le 3d(d-2)$.
\end{proof}

Let $\pr_q\colon X\to \bbP^2$ be the projection of $X$ to a general plane.  The ramification curve $\Ram(q)$ of the projection is equal to $X\cap P_q(X)$.  Its degree is equal to $d(d-1)$.  Using the well-known formula for the arithmetic genus of a complete intersection of hypersurfaces, we obtain that its arithmetic genus is equal to $1+\half d(d-1)(2d-5)$.  The branch curve $B(q)$ of $\pr_q$ is equal to the projection of $\Ram(X)$; hence it is of degree $d(d-1)$. If $\Ram(q)$ is smooth, then it is isomorphic to the normalization of $B(q)$. A bitangent line containing the center of the projection is a secant line of $\Ram(X)$; hence its projection has a singular point of $B(q)$ with at least two different branches. For a quartic surface, it is an ordinary node.  A flex line from $\Flex(X)$ containing $q$ is a tangent line to $P_q(X)$.  Its projection is a singular point with at least two equal branches.

We use the following theorem due to Valentine and Viktor Kulikov \cite[Theorem~0.1]{Kulikov}. 

\begin{theorem}
Assume that $\mathrm{char}(\Bbbk) = 0$ and $X$ has only rational double points as singularities. Let $\pr_q\colon X\to \bbP^2$ be a general projection of\, $X$.  Then, $\pr_q^{*}(B(q)) = 2\Ram(q)+C$, where both $\Ram(q)$ and $C$ are reduced. The projection of a singular point of type $A_n,D_n,E_n$ is a simple singular point of\, $B(q)$ of type $a_n,d_n,e_n$. The other singular points of $B(q)$ are ordinary nodes and ordinary cusps; they are the projections of the bitangents lines and flex lines of\, $X$.
\end{theorem}

\begin{corollary}\label{bidegrees2}
Assume that $\mathrm{char}(\Bbbk) = 0$ and $X$ is smooth or has only ordinary double points. Then, the bidegrees of\, $\Bit(X)$ and $\Flex(X)$ are equal to $(\half d(d-1)(d-2)(d-3),\half d(d-2)(d^2-9))$ and $(d(d-1)(d-2),3d(d-2))$, respectively.
\end{corollary}

\begin{proof}
We choose $q$ general enough such that no bitangent or flex line through $q$ passes through a singular point of $X$.  The ramification curve $\Ram(q)$ of the projection $\pr_q$ is equal to the intersection $X\cap P_q(X)$ of~$X$ with its polar with respect to $q$. We know that the Jacobian ideal of an ordinary double point $x\in X$ is the maximal ideal $\mathfrak{m}_{X,x}$. This easily implies that $P_q(X) = V(\sum_{i=0}^3a_i\frac{\partial F_d}{\partial x_i})$ is smooth at $x$ and that under the projection $\pr_q$, the double points of $\Ram(q)$ are mapped bijectively to ordinary double points of $B(q)$ different from the nodes and cusps coming from bitangent and flex lines.

Since the curve $\Ram(q)$ is a complete intersection of degrees $(d,d-1)$, its arithmetic genus $p_a$ is equal to $ 1+\half d(d-1)(2d-5)$. The projection $\pr_q$ defines a birational isomorphism between the curves $\Ram(q)$ and $B(q)$. The geometric genus $g$ of $B(q)$ is equal to $\half (d(d-1)-1)(d(d-1)-2)-\delta-\kappa -\delta_0$, where $\delta$ is the number of bitangent lines, $\kappa$ is the number of flex lines, and $\delta_0$ is the number of nodes of $\Ram(q)$.  This gives an inequality
\begin{align*}
p_a &= 1+\half d(d-1)(2d-5)\\
&= \half (d(d-1)-1)(d(d-1)-2)-\half d(d-1)(d-2)(d-3)-d(d-1)(d-2)\\
&\ge g \ge \half (d(d-1)-1)(d(d-1)-2)-\delta-\kappa
\end{align*}
that implies the inequality
\begin{equation}\label{ineq1}
\delta+\kappa \ge \half d(d-1)(d-2)(d-3)+d(d-1)(d-2).
\end{equation}
Since $\delta$ is equal to the number of bitangents of $X$ dropped from $q$, we obtain that $\delta$ is the order of $\Bit(X)$. Similarly, we get that $\kappa$ is the order of the $\Flex(X)$. Applying Propositions~\ref{bidegrees} and~\ref{flexes}, we obtain $\delta \le \half d(d-1)(d-2)(d-3)$ and $\kappa\le d(d-1)(d-2)$. It follows from \eqref{ineq1} that $\delta = \half d(d-1)(d-2)(d-3)$ and $\kappa = d(d-1)(d-2)$.

In this proof, we silently assumed that the curves $\Ram(q)$ and $B(q)$ are irreducible.  One can treat the reducible case in a similar manner by dividing the bitangents and flex lines into subsets corresponding to irreducible components of the curves.  We leave it to the reader to finish the proof.
\end{proof}

\section{Quartic surfaces: \texorpdfstring{$\boldsymbol{p\ne 2}$}{p different from 2}}\label{sec2}

In the special case $d = 4$, we expect that, for a smooth or nodal quartic surface in characteristic zero, the bidegree $(m,n)$ of $\Bit(X)$ is equal to $(12,28)$.  Let $S$ be an irreducible component of $\Bit(X)$, and let $q_S\colon Z_S\to S$ be the restriction of the tautological line bundle $q\colon Z\to G_1(\bbP^3)$ to the congruence $S$. The pre-image $\tilde{X}$ of $X$ under the projection $p_S\colon Z_S\to \bbP^3$ is a double cover of $S$ defined by the projection $q_S$:
$$
\xymatrix{&Z_S\ar[dl]_{p_S}\ar[dr]^{q_S}&\\
\bbP^3&&S\rlap{.}}
$$
Thus, any irreducible component $S$ of $\Bit(X)$ with non-zero order and non-zero class defines a birational involution of the surface $\tilde{X}$ with quotient isomorphic to $S$.  In the case when the projection $p_S$ restricts to a birational isomorphism $\tilde{X}\to X$, we obtain a birational involution $\sigma_S$ of $X$.  The set of fixed points of $\sigma_S$ is equal to the pre-image of the locus of points on $X$ such that there exists a line intersecting $X$ at this point with multiplicity $4$.

\begin{example}
  It follows from Kummer's classification of congruences of order $2$ without fundamental curves that the moduli space of quartic surfaces with $16\ge \mu \ge 11$ nodes contains an irreducible component such that a general surface from this component has reducible congruence $\Bit(X)$.  Some of the irreducible components are confocal congruences, \textit{i.e.}, congruences of the same bidegree that share the same focal surface (for example, as we will explain below, the first case in the following Table~\ref{table} is a quartic del Pezzo surface).  The numbers of irreducible components and their bidegrees are given in Table~\ref{table} below (see \cite{Kummer}, \cite[Artikel~348]{Sturm}).
  
\begin{table}[ht]
 \begin{center}
 \begin{tabular}{|c|c|c|}
   \hline  
 $\mu$&Bidegree&Number\\ \hline 
  $16$&$(2,2)$&6\\
 &$(0,1)$&16\\
\hline
 $15$&$(2,3)$&6\\
 &$(0,1)$&10\\ \hline
 $14$&$(2,4)$&4\\
 &$(0,1)$&6\\ 
 &$(4,6)$&1\\ \hline
 $13$&$(2,5)$&3\\
 &$(0,1)$&3\\ 
&$(6,10)$&1\\ \hline
 $12_{\rm I}$&$(2,6)$&2\\
&$(0,1)$&1\\
&$(8,15)$&1\\ \hline
$12_{\rm II}$&$(2,6)$&3\\
&$(6,10)$&1\\ \hline
$11$&$(2,7)$&1\\ 
&$(10,21)$&1\\ \hline
\end{tabular}
\caption{Irreducible components of $\mathrm{Bit}(X)$}\label{table}
\end{center}
\end{table}
\end{example}

The first case is a Kummer quartic surface $X$.  It follows that $X$ admits six involutions with quotient isomorphic to a congruence of bidegree $(2,2)$.  This is a quartic del Pezzo surface embedded in $G_1(\bbP^3)$ via its anti-canonical linear system.  Each pair from the six congruences intersects at four points. The branch curve has $28 = 16+12$ ordinary nodes and $24$ cusps.  Its geometric genus is equal to $3$. The ramification curve $\Ram(q)$ is a curve on $X$ of degree $12$ with $16$ nodes; its geometric genus is also equal to $3$.

We have already mentioned in the introduction that a Kummer surface admits $16$ planes cutting the surface along a smooth conic.  Recall that a Kummer quartic surface is the quotient of the Jacobian of a smooth curve of genus $2$ by the inversion automorphism.  The sixteen double conics (tropes) are the images of the theta divisor and its translations by sixteen $2$-torsion points.  This is a classical, well-known fact that can be found in any treatment of Kummer surfaces (see, for example, \cite[Section~10.3]{DoC}, \cite{Hudson1}).  Let us explain how to see the six irreducible components of bidegree $(2,2)$.  Recall that $X$ admits a smooth model~$\tilde{X}$ as a surface of degree $8$ in $\bbP^5$ (see \cite[Section~10.3.3]{DoC}).  The surface $\tilde{X}$ is a complete intersection of three quadrics in $\bbP^5$ which is a $K3$ surface. One can choose projective coordinates such that $\tilde{X}$ is given by the following equations:
\begin{equation}
\sum_{i=1}^6 z_i^2 = \sum_{i=1}^6 a_iz_i^2 = \sum_{i=1}^6 a_i^2 z_i^2 =0,
\end{equation}
where $a_1,\ldots, a_6$ are distinct constants. The projective transformations
$$
[z_1,z_2,z_3,z_4,z_5,z_6]\longmapsto [\pm z_1,\pm z_2, \pm z_3, \pm z_4, \pm z_5, \pm z_6]
$$
induce automorphisms of $\tilde{X}$ which generate a $2$-elementary group of order $2^5$. The group of such automorphisms contains a subgroup of index $2$ that consists of the identity and $15$ projective involutions that restrict to symplectic automorphisms of $\tilde{X}$.\footnote{This means that the involution leaves invariant a non-zero regular $2$-form on the surface, or, equivalently, the set of its fixed points is non-empty and finite.}  They are the involutions that change two or four signs at the coordinates, and non-symplectic involutions are the ones that change one, three, or five signs.  The quotient surface of $\tilde{X}$ by a symplectic automorphism is a K3 surface with rational double points.  The remaining $16$ involutions are decomposed into two sets.  One set consists of ten involutions that change exactly three signs at the coordinates.  They are fixed-point-free with  orbit space isomorphic to an Enriques surface.  Another set consists of six involutions $g_i$ that change only one coordinate, with  orbit space isomorphic to a del Pezzo surface of degree $4$. These are the involutions we are interested in.

The first quadric $V(\sum z_i^2)$ can be identified with the Grassmannian quadric $G_1(\bbP^3)$ written in Klein coordinates $z_i$ corresponding to six apolar line complexes (see, for example, \cite[Section~10.2]{DoC}). It is known that any automorphism of $G_1(\bbP^3)$ comes from a projective collineation or correlation. The involutions $g_i$ preserve the apolar line complexes $V(z_i)$ and, hence, come from a correlation. They transform points to planes and, therefore, act on $G_1(\bbP^3)$ by transforming the plane of lines through a point (an $\alpha$-plane) to the plane of lines contained in a given plane (a $\beta$-plane).

\begin{lemma}\label{octicmodel}
Let $g_i$ be considered as a birational involution of\, $X$.  Then, the closure of lines spanned by the orbits of $g_i$ is an irreducible component of\, $\Bit(X)$.
\end{lemma}

\begin{proof}
First, we recall the relation between $X$, $\tilde{X}$, and the quadratic line complex $Y = V(\sum a_iz_i^2)\cap G_1(\bbP^3)$ (see, \textit{e.g.}, \cite[Section~10.3.3]{DoC}).  For a point $x\in \bbP^3$, let $\Omega(x) \subset G_1(\bbP^3)$ be the $\alpha$-plane consisting of lines passing through $x$, and for a plane $\Pi \subset \bbP^3$, denote by $\Omega(\Pi)$ the $\beta$-plane consisting of lines lying on $\Pi$.  The Kummer quartic surface $X$ is the set of $x\in \bbP^3$ such that $\Omega(x)\cap Y$ is a singular conic in $\Omega(x)$.  The surface $\tilde{X}$ sits in $G_1(\bbP^3)$ as a singular surface of the quadratic line complex; that is, it parameterizes the lines that are singular points of the intersection of $Y$ with $\alpha$-planes.  We can identify each point $x\in X$ with the $\alpha$-plane $\Omega(x)$.  An automorphism $g_i$ acts on $X$ by sending $\Omega(x)$ to $\Omega(\Pi)$.  Let $x\in X$ be a general point, and let $\Pi = \bbT_x(X)$ be the embedded tangent space of $X$ at $x$.  As we explained above, $g_i$ interchanges the two families $\{\Omega(x)\}_{x\in \bbP^3}$ and $\{\Omega(\Pi)\}_{\Pi\subset \bbP^3}$ of planes, and hence 
$$
g_i(\Omega(x))=\Omega(\Pi'),\quad g_i(\Omega(\Pi))=\Omega(x')
$$
for some $x' \in \bbP^3$ and $\Pi'\subset \bbP^3$.  Since $g_i$ acts on $X$ as a birational automorphism, $x'\in X$ and $\Pi' = \bbT_{x'}(X)$.  Since any point of $H_i = V(z_i)$ is fixed by $g_i$,
\begin{gather*}
\emptyset \ne \Omega(x)\cap H_i = g_i(\Omega(x)\cap H_i)= \Omega(\Pi')\cap H_i,\\
\emptyset \ne \Omega(\Pi)\cap H_i = g_i(\Omega(\Pi)\cap H_i)= \Omega(x')\cap H_i.
\end{gather*}
These imply that $x\in \Pi'$ and $x'\in \Pi$.  Since the line $\ell = \la x, x'\ra$ satisfies $\ell\subset \Pi$ and $\ell\subset \Pi'$, it is tangent to $X$ at $x$ and $x'$; that is, it is a bitangent line of $X$.  Thus, $\ell$ can be identified with the pair $\{x, x'= g_i(x)\}$, that is, a point of the orbit space $\tilde{X}/(g_i)$.
\end{proof}

The quotient surface $\tilde{X}/(g_i)$ is given by
$$
\sum_{j\ne i} (a_j-a_i)z_j^2 = \sum_{j\ne i} \left(a_j^2-a_i^2\right)z_j^2=0; 
$$
it is a quartic del Pezzo surface $\calD$. 

Also, note that the fixed locus of the involution $g_i$ is a canonical curve $C$ of genus $5$ given as the intersection of three diagonal quadrics.  The curve $C$ has a special group of automorphisms isomorphic to the $2$-elementary group $2^4$. It is classically known as a \emph{Humbert curve} of genus $5$. The image $B$ of $C$ in the quartic del Pezzo surface $\calD$ belongs to $|-2K_{\calD}|$.  The curve $B$ is the branch curve of the double cover $\tilde{X}\to \calD$. It has the distinguished property that any of the $16$ lines on $\calD$ splits into a pair of lines on $\tilde{X}$.

\section{Birational involutions of a quartic surface} \label{involutions}
Let $\sigma$ be a birational involution of a quartic surface $X$. The closure of lines spanned by orbits of $\sigma$ is an irreducible congruence $S(\sigma)$ of lines.  Fix a general point $P\in \bbP^3$; the order of $S(\sigma)$ is equal to the cardinality of the set $\{x\in \bbP^3: P\in \la x,\sigma(x)\ra\}$.

Suppose that $\sigma$ is the restriction of a Cremona involution $T$.  Then, the set of points $x\in \bbP^3$ such that $P\in \la x,T(x)\ra$ is classically called an \emph{isologue} of $T$, and $P$ is its center.  It is given by the condition that
$$
\rank\begin{pmatrix}a&b&c&d\\\
x_0&x_1&x_2&x_3\\
y_0&y_1&y_2&y_3\end{pmatrix} < 3,
$$
where $P = [a,b,c,d], x = [x_0,x_1,x_2,x_3], T(x) = [y_0,y_1,y_2,y_3]$.  It is expected to be a curve of degree $2k+1$, where $k$ is the algebraic degree of $T$; see \cite[Section~IX.21, p.~175]{Hudson2}. However, it also includes the locus of fundamental points of $T$, as well as the closure of the locus of fixed points of $T$.  The class of $S(\sigma)$ is expected to be equal to the degree of $\Pi\cap T(\Pi)\cap X$ for a general plane $\Pi$, which is $4k$.
 
Let $\sigma$ be a birational involution of a quartic surface. The rational map
$$
\phi_\sigma\colon X/(\sigma)\longdashrightarrow G_1\left(\bbP^3\right),\quad x\longmapsto \la x,\sigma(x)\ra
$$ 
is of degree $1$ or $2$, since a general line intersects $X$ with multiplicity $4$.  There are three possible scenarios:
\begin{enumerate}[label=(\roman*), ref=\roman*]
\item\label{scen1} $\phi_\sigma$ is of degree $2$;
\item\label{scen2} $\phi_\sigma$ is of degree $1$, a general line $\ell_x = \la x,\sigma(x)\ra$ intersects $X$ 
at two points and 
two fixed points of $\sigma$;
\item\label{scen3} $\phi_\sigma$ is of degree $1$, a general line $\ell_x$ is a bitangent line of $X$.
\end{enumerate}

In the last case, we say that $\sigma$ is a \emph{bitangent involution}.

Let $S(\sigma)\subset G_1(\bbP^3)$ be the congruence of lines defined as the closure of the image of the rational map $\phi_\sigma$.

\begin{example}
Here, we give an example of an involution of type~\eqref{scen2}.  Let $\sigma$ be the restriction of a projective involution $T$ in $\bbP^3$.  Suppose that $p \ne 2$ and the fixed locus of $T$ is the union of two skew lines $\ell_1,\ell_2$.  It is clear that the lines $\ell_x$ are invariant lines of $T$.  Hence, each $\ell_x$ has two fixed points on it, one on $\ell_1$ and one on $\ell_2$.  This shows that the congruence $S(\sigma)$ is of bidegree $(1,1)$, isomorphic to a smooth quadric.  It is known that the set $X^\sigma$ of fixed points of an involution $\sigma$ of a K3 surface consists of $8$ isolated fixed points if $\sigma$ is a symplectic involution and that it does not have isolated fixed points if $\sigma$ is an anti-symplectic involution. This implies that either $T$ intersects $X$ at no more than eight points (one can show that there are exactly eight intersection points since none of the lines can pass through a singular point; however, we will not need this fact), or $T$ is contained in $X$. In the former case, each line $\ell_x$ has two fixed points, one on $\ell_1$ and another on $\ell_2$. Thus, the degree of $\phi_\sigma$ is equal to~$1$.

In the second case, assume $X$ is smooth and $X^\sigma$ consists of two lines $\ell_1$ and $\ell_2$.  Then, $Y = X/(\sigma)$ is also smooth.  The cover $X\to Y = X/(\sigma)$ is ramified over $\ell_1+\ell_2$, and its branch curve is the union of two disjoint smooth rational curves $C_1, C_2$ with self-intersection $-4$. By the adjunction formula, $|-K_Y| = \emptyset$ and $|-2K_Y| = \{C_1+C_2\}$.  Thus, $|-K_Y| = \emptyset$ and $|-2K_Y|\ne \emptyset$; \textit{i.e.}, $Y$ is a Coble surface obtained by blowing up ten points on a smooth quadric $Q$ (see \cite[Section~9.1]{Enriques2}), the eight intersection points of two nodal quartic curves $\bar{C}_1,\bar{C_2}$ of bidegree $(2,2)$ and the two nodes.  The pre-images of the ten exceptional curves on $Y$ in $X$ are ten invariant lines lying in $X$ that contain infinitely many orbits of $\sigma$.

Here is a concrete example: Let $X = V(F)$, where
$$
F = x_3^3x_0 + x_3^3x_1 + x_3x_0^3 + x_3x_0x_2^2 + x_3x_1^3 + x_0^3x_2 + x_0^2x_1x_2 + x_0x_2^3 + x_1^3x_2 + x_1x_2^3
$$
and $\sigma\colon [x_0,x_1,x_2,x_3]\mapsto [-x_0,-x_1,x_2,x_3]$.  The set of fixed points of $\sigma$ consists of two skew lines $\ell_1 = V(x_0,x_1)$ and $\ell_2 = V(x_2,x_3)$. Plugging in the parametric equation $[(s-t)u_0,(s-t)u_1,(s+t)u_2,(s+t)u_3]$ of a line $\ell_x$, where $x = [u_0,u_1,u_2,u_3]$ is not a fixed point of $\sigma$, we obtain the expression
$$
\left(s^2-t^2\right)\left(F(u_0,u_1,u_2,u_3)s^2+2G(u_0,u_1,u_2,u_3)st+F(u_0,u_1,u_2,u_3)t^2\right),
$$
where 
$$
G= -u_0^3u_2 - u_0^3u_3 - u_0^2u_1u_2 + u_0u_2^3 + u_0u_2^2u_3 + u_0u_3^3 - u_1^3u_2 - u_1^3u_3 + u_1u_2^3 + u_1u_3^3.
$$ 
Thus, the union of lines contained in $X$ and intersecting the skew lines $\ell_1$ and $\ell_2$ is equal to $X\cap V(G')$, where $G'$ is obtained from $G$ by replacing $u_i$ with $x_i$.

We can check (using Maple software) that if $\mathrm{char}(\Bbbk)\ne 5, 643$, the surfaces $V(F)$ and $V(G')$ intersect transversally at ten lines and intersect at the two lines of fixed points with multiplicity $3$. This shows that $Y = X/(\sigma)\to S(\sigma) = Q$ is the blowing down of ten $(-1)$-curves on the Coble surface $Y$.
\end{example} 

\begin{example}
Here is an example of a bitangent involution. We refer the reader to \cite{Cossec} or \cite[Section~7.4]{Enriques2}.  Assume $\cha(\Bbbk) \ne 2$. Let $W$ be a general web of quadrics in $\bbP^3$, and let $\calD(W)\subset W$ be the set of singular quadrics in $W$, classically known as a Cayley quartic symmetric (see \cite[Definition~2.1.1]{Cossec}).  Choose a basis $\{q_0, q_1, \ldots, q_3\}$ of $W$, and let $q(\lambda)=\sum_i\lambda_i q_i \in W$ for $\lambda =(\lambda_i)$ a coordinate of $W$.  Here, we identify a quadric in $\bbP^3$ and a symmetric $4\times 4$-matrix.  Then, $\calD(W)$ is a quartic surface in $W \cong \bbP^3$ defined by $\det(q(\lambda))=0$.  Another attribute of a web of quadrics is the Steinerian (or Jacobian) surface in~$\bbP^3$, the locus of singular points of quadrics from $W$. This surface is also a quartic surface, but it lies in the original~$\bbP^3$ but not in $W$ (see \cite[Section~1.1.6]{DoC}).

A Reye line is a line in $\bbP^3$ which is contained in a pencil from $W$. It is known that the set of Reye lines forms an irreducible congruence $\Rey(W)$ of bidegree $(7,3)$, see \cite[Proposition~3.4.2]{Cossec}, a \emph{Reye congruence}, and each Reye line is bitangent to the Steinerian surface $X$ of $W$.
 
The surface $X$ admits a fixed-point-free involution $\sigma$ such that $S(\sigma) = \Rey(W)$.  The involution is defined by $x\mapsto \cap_{Q\in W}P_x(Q)$, where $P_x(Q)$ is the first polar of $Q$ with pole at $x$.

One can also consider the surface of bitangents $\Bit(\calD(W))$ of the quartic symmetroid $\calD(W)$. The generality assumption on $W$ implies that $\calD(W)$ does not contain lines.  So, if $\cha(\Bbbk) = 0$, Corollary~\ref{bidegrees2} implies that the bidegree of $\Bit(\calD(W))$ is equal to $(12,28)$.  It follows from \cite[Theorem~7.4.7]{Enriques2} that it is an irreducible surface with normalization isomorphic to the Reye congruence $\Rey(W)$.  So, $\Rey(W)$ admits two birational models as congruences of bidegrees $(7,3)$ and $(12,28)$.
\end{example}

A point of a congruence of lines is called a fundamental point (in classical terminology, a singular point) if the set of rays passing through this point contains a $1$-dimensional component.  In our case of the congruence $\Bit(X)$, a fundamental point must be a singular point of $X$. For example, if $X$ is a Kummer quartic, the rays through its singular point sweep a trope.  It is classically known that all rays of a congruence of lines of bidegree $(m,n)$ with only isolated fundamental points are tangent to the focal surface $\Phi$ of the congruence of degree $2m+2g-2$, where $g$ is the genus of a general hyperplane section of the congruence (see \cite{Sturm,J,DoReid}).  Since $m = 12$, the degree is at least $22$. However, $\Phi$ could be reducible, and $X$ is its irreducible component. So, the existence of a bitangent involution on $X$ implies that the focal surface of $\Bit(X)$ is reducible.
 
Table~\ref{table} gives examples of families of quartic surfaces which admit bitangent involutions with one irreducible component of order $2$.

It is clear that, if a quartic surface $X$ admits a bitangent involution $\sigma$, then the congruence of lines $S(\sigma)$ is an irreducible component of $\Bit(X)$.

\begin{remark}
As we explained, any birational involution $\sigma$ of a quartic surface $X$ defines an irreducible congruence of lines $S(\sigma)$.  If the order of $S(\sigma)$ is equal to~$1$, the involution $\sigma$ lifts to a birational involution~$T$ of $\bbP^3$. Indeed, given a general point $P$ in $\bbP^3$, there is a unique line passing through $P$ and spanned by an orbit of two points $x,\sigma(x)$ on $X$. We define $T(P)$ to be the fourth point on this line such that the pairs $(x,\sigma(x))$ and $(P,T(P))$ are harmonically conjugate.
\end{remark}

\section{Quartic surfaces: \texorpdfstring{$\boldsymbol{p = 2}$}{p=2}}\label{sec3}
In this section, we assume that $\cha(\Bbbk) =2$. We continue to assume that $X$ is a normal quartic surface.

It is obvious that, in characteristic $p\ne 2$, the bitangent surface of a normal quartic surface $X$ does not contain $\alpha$-planes (otherwise, the projection of $X$ from some point is ramified at every point).  This is not true anymore  in characteristic $p = 2$.

A point $x$ in $\bbP^3$ is called an \emph{inseparable projection center} of a normal surface $X$ if the projection map with  center at $x$ is inseparable. It is clear that the set of lines passing through an inseparable projection center is an $\alpha$-plane contained in $\Bit(X)$. Conversely, if $\Bit(X)$ contains an $\alpha$-plane of lines through a point $x\in \bbP^3$, the point $x$ is an inseparable projection center of $X$.

\begin{proposition}
The set of inseparable projection centers of a normal quartic surface in characteristic $2$ is finite. Any inseparable projection center contained in $X$ is a singular point of $X$.
\end{proposition}

\begin{proof}
  Since the set of points $x\in X$ such that the line $\la q,x\ra$ is tangent at $x$ is equal to the intersection $X\cap P_q(X)$, see \cite[Theorem~1.1.5]{DoC}, a point $q= [a_0,a_1,a_2,a_3]$ is an inseparable center of a quartic surface $X = V(F)$ if and only if the polar $P_q(X) = V(\sum a_i\frac{\partial F}{\partial x_i})$ is equal to $\bbP^3$. Suppose the set of inseparable projection centers contains a curve $Z$ of degree at least~$2$ (not necessarily irreducible). Then, we can choose three non-collinear points on it. Choose projective coordinates such that these points are $[1,0,0,0],[0,1,0,0]$, and $[0,0,1,0]$. Then, the partials $\frac{\partial F}{\partial x_i}$ are equal to $0$ for $i= 0,1,2$.  This means that $F$ contains $x_0,x_1,x_2$ in even power.  But, then, it must contain $x_3$ in even power. Hence, $F$ is a square. So, $Z$ must be a line. In this case, $\Bit(X)$ contains a special hyperplane section of the Grassmannian $G_1(\bbP^3)$ that consists of lines intersecting $Z$. A general plane section of $X$ intersects $Z$ and, hence, contains a curve of bitangent lines. However, a smooth quartic curve contains only finitely many bitangents (in fact, at most seven; see Proposition~\ref{wall}).
  Thus, a general plane section of $X$ is singular, contradicting the Akizuki--Bertini theorem (see \cite[Remark~8.18.1]{Hartshorne}).  This proves that the set of inseparable centers is finite.
 
Suppose $q\in X$ is an inseparable projection center. Then, we may choose projective coordinates such that $q = [0,0,0,1]$ and the equation of $X$ can be written in form $\sum A_{4-k}(x,y,z)w^k = 0$, where $A_i$ is a homogeneous polynomial in $x, y, z$ of degree $i$.  Since $q\in X$, we have $A_0 = 0$, and since $P_q(X) = 0$, we have $A_1 = 0$. Thus, $q$ is a singular point.
\end{proof}

We can extend the proof of Proposition~\ref{bidegrees} to the case of characteristic $2$ to obtain the following corollary.

\begin{corollary}
The variety $\Bit(X)$ is a congruence of lines; \textit{i.e.}, each irreducible component of\, $\Bit(X)$ is a surface.
\end{corollary}

Recall that a universal binary form of degree $d$ is a homogeneous polynomial of degree $d$ in two variables whose coefficients are algebraically independent over a field $k$..
 
\begin{proposition}\label{discri}
Assume $\mathrm{char}(\Bbbk) = 2$. Then, the discriminant $\calD(d)$ of the universal binary form of degree $d$ is a square of a homogeneous polynomial of degree $d-1$.
\end{proposition}

\begin{proof}(Supplied by  G.~Kemper)
Take a univariate polynomial $f= a_0x^d + \cdots + a_1x + a_d$ whose coefficients are indeterminates over the field $\bbF_2$. Let $D = D(a_0,\ldots,a_d)$ be the discriminant of $f$.  We know that $D$, as a function in roots considered as indeterminants $y_1,\ldots,y_d$, is equal to the square of $P = \prod_{i < j} (y_i - y_j)$. Since $\cha(\Bbbk) = 2$, we have $P= \prod_{i < j} (y_i + y_j)$, and hence, $P$ is invariant under the whole symmetric group $\frakS_d$ (and not only under the alternating group $\mathfrak{A}_d$ if $\cha(\Bbbk) \ne 2$).  Since the Galois group of $f$ permutes the roots $r_i$, this means that $P$ lies in $\bbF_2[a_0, \ldots, a_d]$. So, the discriminant of $f$ is a square.
\end{proof}

\begin{corollary}
Suppose $\Bit(X)$ does not contain $\alpha$-planes. Then, the order $m$ of $\Bit(X)$ satisfies
$$
1\le m\le \tfrac{1}{4}d(d-1)(d-2)(d-3).
$$
\end{corollary}

\begin{proof}
The assumption implies that the order of $\Bit(X)$ is equal to the number of bitangents dropped from a general point in $\bbP^3$.  In the proof of Theorem~\ref{bidegrees}, by our assumption, the discriminant of the binary form $g_{d-2}$ is not zero; hence $m\ge 1$.  Now apply Proposition~\ref{discri} and obtain the other inequality.
\end{proof}

\begin{definition}
Let $m(d):=\frac{1}{2}d(d-1)(d-2)(d-3)$ (resp.\ $m(d):=\frac{1}{4}d(d-1)(d-2)(d-3)$) for $p\ne 2$ (resp.\ $p=2$).  We say that a surface $X$ is a \emph{general projection} surface if the order of $\Bit(X)$ is equal to $m(d)$.
\end{definition}

By Corollary~\ref{bidegrees2}, any surface in characteristic zero with ordinary double points as singularities is a general projection surface.

In the next section, we show that this is not true anymore in characteristic $2$.

In the following example of an involution $\sigma$ of type~\eqref{scen2} of a quartic surface in characteristic $2$, all rays from the congruence of lines $S(\sigma)$ are tangent to the surface.

\begin{example}
Let $X = V(F)$ be a quartic surface, where
$$
F= x_0^3x_1+x_1^3x_2+x_2^3x_3+x_3^3x_0.
$$
The surface is invariant with respect to the involution
$$
\sigma\colon [x_0,x_1,x_2,x_3]\longmapsto [x_2,x_3,x_0,x_1].
$$
The set of singular points of $X$ consists of one rational double point $x_0 = [1,1,1,1]$ of type $A_2$ and four ordinary double points $[1,\zeta,\zeta^3,\zeta^2]$, where $\zeta^5 = 1, \zeta\ne 1$.  The fixed locus $X^\sigma$ of $\sigma$ is the line $\ell=V(x_0+x_2,x_1+x_3)$.
 
For any point $x = [a,b,c,d]\in X$ not lying on $\ell$, the line $\ell_x = \la x,\sigma(x)\ra$ contains one extra point $[a+c,b+d,a+c,b+d]\in \ell$. The line $\ell_x$ is tangent to $X$ at this point.  So, $\sigma$ is an involution of type~\eqref{scen2}.
 
Any invariant line on $X$ is contained in a plane $x_3=x_1+t(x_0+x_2)$.  A straightforward computation shows that there are four invariant lines on $X$ corresponding to the parameters $t =0, \infty, e,e^2$, where $e^2+e+1 = 0$.  Under the map $\phi\colon X/(\sigma) \to S(\sigma) $, these lines are blown down to points.  The algebra of invariants $\Bbbk[x_0,x_1,x_2,x_3]^{(\sigma)}$ is generated by the polynomials
$$
p_0= x_0+x_2,\quad  p_1 = x_1+x_3,\quad  p_2 = x_0x_2,\quad  p_3 = x_1x_3,\quad  p_4 = x_0x_1+x_2x_3.
$$
They satisfy the relation 
$$
p_0^2p_3 + p_0p_1p_4 + p_1^2p_2 + p_4^2 = 0,
$$
and embed $\bbP^3/(\sigma)$ into $\bbP(1,1,2,2,2)$ as a weighted homogeneous hypersurface of degree $4$ with double line $V(p_0,p_1,p_4)$. We can write
$$
F = p_4\left(p_3+p_0^2\right)+\left(p_2+p_1^2\right)(p_0p_1+p_4).
$$
Therefore, the image of $X$ in $\bbP^3/(\sigma)$ is the intersection of two hypersurfaces of degree $4$, so it has trivial dualizing sheaf.  It is singular along the line $V(p_0,p_1,p_4)$.  The quotient $ X/(\sigma)$ is isomorphic to the normalization of this surface. Note that the polynomials $p_i$ do not generate the algebra of invariants of the projective coordinate ring of $X$.  In fact, $x_0(x_3^3+x_0x_1)$ and $x_1(x_0^3+x_1^2x_2)$ are invariant modulo $(F)$.

The map $\phi\colon X/(\sigma) \to S(\sigma)$ is just the projection map given by $(p_0,p_1,p_3)$. This shows that the congruence of lines $S(\sigma)$ is isomorphic to $\bbP(1,1,2)$.  The image $\phi(\ell)$ of the line $\ell$ is the singular point $[0,0,1]$ of $\bbP(1,1,2)$.  The images of the four invariant lines are two non-singular points.
\end{example}

\begin{lemma}\label{trope2}
With no assumption on the characteristic, suppose $\Bit(X)$ contains a $\beta$-plane of lines in a plane $\Pi = V(L)$, where $L$ is a linear form.  Then, the equation of\, $X$ can be written in the form
\begin{equation}\label{trope}
Q(x,y,z,w)^2+L(x,y,z,w)F(x,y,z,w) = 0,
\end{equation}
where $Q$ is a quadratic form and $F$ is a cubic form. The singular locus of\, $X$ contains the union of\, $V(L,Q,F)$ and $V(Q)\cap \Sing(V(F))$,
where $\Sing(V(F))$ is the singular locus of\, $V(F)$. In particular, $X$ is a singular quartic surface.
\end{lemma}

\begin{proof}
Let $C = X \cap \Pi$. Any line in $\Pi$ is a bitangent of $X$. Thus, it is bitangent to $C$. This could happen only if $C$ is a double conic, reducible or not. Thus, $\Pi$ is a trope conic, and its equation can be written as in \eqref{trope}.

Without loss of generality, we may assume that $L = x$ is a coordinate plane. Taking the partials, we find the singular locus contains $V(x,Q,F)\cup (V(Q) \cap \textrm{Sing}(V(F)))$.
\end{proof}

Recall that the \emph{$p$-rank} of a smooth curve $C$ of genus $g > 0$ over an algebraically closed field of characteristic $p > 0$ is the $p$-rank of the elementary abelian group $J(C)[p]$ of $p$-torsion points of the Jacobian variety of $C$.  It takes values in the set $[0,g]$. When $p = 2$ and $C$ is a smooth plane quartic, the $p$-rank takes values in $\{3,2,1,0\}$ and can be characterized by the number of bitangents of $C$ being equal to $7, 4, 2, 1$, respectively; see \cite{R}, \cite[Section~3]{SV}.

The proof of the following proposition can be found in \cite[Proposition 1]{Wall} (see also \cite{R}).

\begin{proposition}\label{wall}
Let $C$ be a smooth plane quartic curve over an algebraically closed field of characteristic $2$. Then, it is projectively equivalent to one of the curves
\begin{enumerate}
\item\label{wall-1} $Q(x,y,z)^2+xyz(x+y+z) = 0$, where $Q(1,0,0),\  Q(0,1,0),\  Q(0,0,1),\ Q(1,1,0),\
Q(1,0,1),
Q(0,1,1)$, $Q(1,1,1)\ne 0 $;
\item\label{wall-2} $Q(x,y,z)^2+xyz(y+z) = 0$, where $Q(1,0,0),\  Q(0,1,0),\  Q(0,0,1),\  Q(0,1,1) \ne 0$;
\item\label{wall-3} $Q(x,y,z)^2+xy(y^2+xz) = 0$, where $Q(1,0,0),\  Q(0,0,1)\ne 0$;
\item\label{wall-4} $Q(x,y,z)^2+x(y^3+x^2z)  = 0$, where $Q(0,0,1)\ne 0$.
\end{enumerate}
\end{proposition}

One can check that the number of bitangents is indeed equal to $7,4,2,1$, respectively. More precisely, the bitangents are
\begin{enumerate}
\item $V(x), V(y), V(z), V(x+y), V(y+z), V(z+x), V(x+y+z)$;
\item $V(x), V(y),V(z), V(y+z)$;
\item $V(x), V(y)$;
\item $V(x)$.
\end{enumerate}

The paper \cite{Wall} of Wall also computes the automorphism group of a plane quartic in one of the normal forms~\eqref{wall-1}--\eqref{wall-4}.  This implies that the codimension of the subspace of plane quartics of the form~\eqref{wall-1}--\eqref{wall-4} is equal to $0,1,2,3$, respectively; see \cite{Wall}.

Based on Wall's computations, the following proposition is proved in \cite[Section 2]{Nart}.

\begin{proposition}\label{nart}
A smooth plane quartic curve over any field $k$ $($not necessary algebraically closed\,$)$ of characteristic~$2$ of $p$-rank less than $3$ is isomorphic over $K$ to a curve $Q^2+LF$, where $L$ is a linear form.
\end{proposition}

The next theorem shows, surprisingly for us, that, although the moduli space of non-ordinary (\textit{i.e.}, of $p$-rank less than $3$) curves is of dimension $5$ in the moduli space of all plane quartic curves, a general hyperplane section of a non-singular quartic surface is an ordinary curve.
  
\begin{theorem}\label{$2$-rank0}
Let $X$ be a normal quartic surface over an algebraically closed field of characteristic $2$ whose general hyperplane section is not an ordinary plane quartic.  Then, $\Bit(X)$ contains a $\beta$-plane, and hence, $X$ is singular.
\end{theorem}

\begin{proof}   
Let
$$
F:= \left\{(x,\Pi)\in X\times \check{\bbP}^3:x\in \Pi\right\} \subset \left\{(x,\Pi)\in \bbP^3\times \check{\bbP}^3:x\in \Pi\right\}
$$
be the universal family of plane sections of $X$ considered as the closed subset of the universal family of planes.  Passing to the generic fiber of the second projection, we obtain a plane quartic curve $C_K$ over the field $K$ of rational functions on $\check{\bbP}^3$.  Applying Proposition~\ref{nart}, we find that $C_K = V(Q^2+LF)$, where $L$ is a linear form with coefficients in $K$.  Applying a linear transformation of $\bbP^3$ over $K$, we may assume that the equation of $V(L)$ in $\bbP_K^3$ is $x_0 = 0$, where $(x_0,x_1,x_2,x_3)$ are coordinates in $\bbP_K^3$. This implies that any line in the plane $V(x_0)$ is a bitangent line of $X$.  Applying Lemma~\ref{trope2}, we obtain that $X$ is singular.
  \end{proof}

\begin{example} (Suggested by the referee).
  Consider the surface $X$ with  equation
$$
w^2(ax+by+z)^2+x\left(y^3+x^2z\right) = 0.
$$
The surface $X$ has a unique (non-rational) singular point $[0,0,0,1]$ and an ordinary node $[0,0,1,0]$.  The point $[0,0,0,1]$ is a unique inseparable projection center of $X$.  In fact, take a general point $P = [x_0,y_0,z_0,w_0]$ in $\bbP^3$. Substituting the parametric equation $[x,y,z,w] = [sx_0+tu_0,sy_0+tu_1,sz_0+tu_2,sw_0+tu_3]$, we find one bitangent line that connects $P$ with the point $[0,0,0,1]$.  Thus, $\Bit(X)$ contains an $\alpha$-plane. On the other hand, a general hyperplane section $w= \alpha x+\beta y + \gamma z$ is the case of Proposition~\ref{wall}; that is, it is a plane quartic curve with 2-rank zero.  By Proposition~\ref{$2$-rank0}, it contains a $\beta$-plane. Thus, $\Bit(X)$ contains the union of an $\alpha$-plane and a $\beta$-plane.  Further computation shows that there is nothing else.
\end{example}

\section{Kummer quartic surfaces in characteristic 2, ordinary case}\label{sec4}
Kummer quartic surfaces in characteristic $2$ are divided into three classes according to curves of genus $2$ being ordinary, $2$-rank $1$, or supersingular (see, \textit{e.g.}, \cite{KK}).  In this section, we discuss the simplest case, an ordinary Kummer quartic surface.  The Kummer quartic surface $X$ in characteristic $2$ associated with an ordinary genus $2$ curve $C$ is given by
\begin{equation}\label{eq}
X=V\left(a\left(x^2y^2+z^2w^2\right) + b\left(x^2z^2 + y^2w^2\right)+c\left(x^2w^2+y^2z^2\right)+ xyzw\right),
\end{equation}
where $a, b, c$ are non-zero constants. They coincide with the coefficients of the Igusa canonical model of $C$; see \cite{LP,LP2,KK}.  The Kummer quartic surface $X$ has four singular points
$$
p_1= [1,0,0,0],\quad p_2=[0,1,0,0],\quad p_3=[0,0,1,0],\quad p_4= [0,0,0,1],  
$$ 
all of which are rational double points of type $D_4$, see \cite{Sh,Ka}, and $X$ has four tropes defined by the hyperplane sections $x=0, y=0, z=0, w= 0$, respectively.

Applying Proposition~\ref{wall}, we find that a general hyperplane section of $X$ is a plane quartic with the $2$-rank equal to $3$. So, we expect that the class of $\Bit(X)$ is equal to $7$. Computing the partial derivatives of the polynomial defining $X$, we find that there are no inseparable projection centers of $X$. So, $\Bit(X)$ does not contain $\alpha$-planes.

\begin{lemma}\label{lemSkew}
Let $S$ be the congruence of lines of bidegree $(1,1)$ of rays intersecting the skew lines $\ell_1 = V(x,y)$ and $\ell_2 = V(z,w)$ $($or $V(x,z)$ and $V(y,w)$, or $V(x,w)$ and $V(y,z))$.  Then, each ray of\, $S$ is a bitangent line of\, $X$.
\end{lemma}

\begin{proof}
It is enough to consider the first pair of skew lines. A line $\ell$ passing through a point $q = [x_0,y_0,z_0,w_0]$ not on $\ell_1$ or $\ell_2$ is the intersection of two planes $\la \ell_1,q\ra$ and $\la \ell_2,q\ra$. The parametric equation of $\ell$ is
$$
[x,y,z,w] = [sx_0,sy_0,tz_0,tw_0].
$$
The line $\ell$ intersects $\ell_1$ at the point $[0,0,z_0,w_0]$ and intersects $\ell_2$ at the point $[x_0,y_0,0,0]$. Plugging these equations into \eqref{eq}, we get
\begin{equation}\label{square}
a\left(x_0^2y_0^2s^4 + z_0^2w_0^2t^4\right)+\left(b\left(x_0^2z_0^2 + y_0^2w_0^2\right)+ c\left(x_0^2w_0^2+ y_0^2z_0^2\right) +
x_0y_0z_0w_0\right)s^2t^2=0.
\end{equation}
This equation is a square of a quadratic equation, and hence the assertion holds.
\end{proof}

We keep the notation used in the proof of Lemma~\ref{lemSkew}.  If $q$ belongs to $X$, then it corresponds to $[s,t] = [1,1]$, so $[s,t] = [1,1]$ is one of the solutions of~\eqref{square}. The second solution is $[s,t] = [z_0w_0, x_0y_0]$.  This defines an explicit bitangent involution $\sigma_1$ of $X$. We see that it is undefined only at the singular points $p_1,\ldots, p_4$ of $X$.  The involution is the restriction of the Cremona involution
$$
T_1\colon [x,y,z,w]\longmapsto [xzw,yzw,xyz,xyw], 
$$
which is equal to the composition of the standard inversion transformation $T$ and the involution 
$$
g_1\colon [x,y,z,w]\longmapsto [y,x,w,z].
$$ 
Note that $g_1$ is induced from a translation of $J(C)$ by a non-zero 2-torsion point.  There are three non-zero 2-torsion points on $J(C)$, and all three bitangent involutions are  products of $T$ and the involutions induced from the translations.
  
\begin{proposition}\label{pencils}
The surface $X$ admits three bitangent involutions $\sigma_1,\sigma_2,\sigma_3$.  The congruence of lines $S(\sigma_i)$ is equal to the congruence of lines intersecting two skew lines from Lemma~\ref{lemSkew}.  The fixed curve $X^{\sigma_i}$ is an elliptic curve of degree $4$, and it is cut out set-theoretically by the quadric $V(xy+zw)$.
\end{proposition}

\begin{proof}
It is enough to consider the first involution defined by the two lines $\ell_1 = V(x,y)$ and $\ell_2 = V(z,w)$.  We denote it by $\sigma$.  We have already proved the first assertion.  The fixed locus of $T_1$ is the quadric $Q= V(xy+zw)$. The quadric intersects $X$ along a complete intersection of two surfaces:
\begin{eqnarray*}
 &&xy+zw = 0,\\
 &&b\left(x^2z^2 + y^2w^2\right)+ c\left(x^2w^2+ y^2z^2\right) +xyzw = 0.
\end{eqnarray*}
Plugging  $y=zw/x$ into the second equation, we find that the intersection is a quartic curve taken with multiplicity $2$. It passes through the singular points, and these points are non-singular on the quartic.  We can also check (substituting $x = 1$) that its projection to the plane is a smooth cubic curve.  Thus, $C = X^\sigma$ is, set-theoretically, a smooth quartic elliptic curve.
  
Note that this result is analogous to the fact that the fixed locus of any of the six bitangent involutions of a Kummer quartic surface in characteristic different from $2$ is an octic curve cut out by a quartic surface with multiplicity $2$.
\end{proof}

Let us now look at the birational map 
$$
\phi_\sigma:Y= X/(\sigma)\longdashrightarrow S(\sigma) \cong \bbP^1\times \bbP^1.
$$
Recall that $\phi_\sigma$ is not defined at the singular points of $X$.  The pencil of planes $\ell_1^\perp$ (resp.\ $\ell_2^\perp$) with  base line $\ell_1$ (resp.\ $\ell_2$) cuts out a pencil of plane quartic curves in $X$. The involution $\sigma$ acts identically on the parameters of the pencil.  Let
$$
\pi\colon\tilde{X}\lra X
$$ 
be the minimal resolution of $X$. The involution $\sigma$ lifts to a biregular involution $\tsigma$ of 
$\tX$. 
  
\begin{lemma}\label{pencilD6A1}
The pencils $\ell_1^\perp$ and $\ell_2^\perp$ define two invariant elliptic pencils $|F_1|$ and $|F_2|$ on $\tX$.  Each pencil has four reducible fibers: two fibers of type $\tilde{D}_6$ and two fibers of type $\tilde{A}_1^*$ $($of Kodaira's type {\rm III}$)$.
\end{lemma}
  
\begin{proof} 
We consider only the pencil defined by $\ell_1^\perp = V(y+tx)$ with parameter $t\in \bbP^1$.  We check that, for a general parameter $t$, the plane quartic $X\cap V(y+tx)$ has two cusps at $p_3$ and $p_4$; hence its geometric genus is equal to $1$.  The cusps are the base points of the pencil of quartics. There are special values of the parameter $t$: $0, \infty, \sqrt{c/b},\sqrt{b/c}$. For $t=0$ (resp.\ $t=\infty$), the quartic is a double conic; it passes through the third singular point $p_1\in \ell_2$ (resp.\ $p_2\in \ell_2$).  The conic is one of the four \emph{trope-conics} on $X$, \textit{i.e.}, conics cut out set-theoretically by one of the trope planes:
\begin{eqnarray*}
C_{123}& =& V\left(\sqrt{a}xy+\sqrt{b}xz+\sqrt{c}yz\right),\\
C_{124}& =& V\left(\sqrt{a}xy+\sqrt{b}yw+\sqrt{c}xw\right),\\
C_{134}& =& V\left(\sqrt{a}zw+\sqrt{b}xz+\sqrt{c}xw\right),\\
C_{234}& =& V\left(\sqrt{a}zw+\sqrt{b}yw+\sqrt{c}yz\right).
\end{eqnarray*}

Let $E_i=E_0^{(i)}+E_1^{(i)}+E_2^{(i)}+E_3^{(i)}$ for $i = 1, 2, 3, 4$ be the exceptional curves over the singular points $p_i$ of~$X$, where $E_0^{(i)}$ is the central component.  We can check that the proper transform of $V(y)$ intersects one of the components $E_i^{(j)}$ for $i\ne 0$ and $j = 1,2,3,4$.  The corresponding fiber of $|F_1|$ is of type $\tilde{D}_6$.  If $t = a/b$, the quartic has one cusp at one of the base points and a cusp followed by an infinitely near node at another base point.  The corresponding fiber is of type $\tilde{A}_1^*$.
\end{proof}
  
Let us look at the orbit space $\tilde{Y} = \tX/(\tsigma)$.  Recall that each singular point of $X$ is a rational double point of type $D_4$ with  exceptional curve $E_0^{(i)}+E_1^{(i)}+E_2^{(i)}+E_3^{(i)}$ over $p_i$, where $E_0^{(i)}$ is the central component.  Observe that the transformation $T_1$ blows down each conic $C_{ijk}$ to a singular point:
$$
C_{123}\lra p_3,\quad  C_{124}\lra p_4,\quad  C_{134}\lra p_1, \quad C_{234}\lra p_2.
$$
The involution $\tsigma$ interchanges the proper transform $\tilde{C}_{123}$ of $C_{123}$ with the central component $E_0^{(3)}$, and similarly for other trope conics.  It acts on the corresponding reducible fiber via a symmetry of its dual graph.

Now, we are ready to describe the birational morphism
$$
\tilde{\phi}_\sigma\colon\tilde{Y}\lra S(\sigma).
$$
Let $|\frakf_1|$ and $|\frakf_2|$ be the rulings of the quadric $Q=S(\sigma)$ corresponding to the family of planes $\ell_1^\perp$ and~$\ell_2^\perp$. Their pre-images under $\tX \to \tilde{Y} \to Q$ are the two elliptic pencils from Lemma~\ref{pencilD6A1}.  Let $L_0^{(1)},L_\infty^{(1)}$ (resp.\ $L_0^{(2)},L_\infty^{(2)}$) be the lines from $|\frakf_1|$ (resp.\ $|\frakf_2|$) corresponding to the planes $V(x), V(y)\in \ell_1^\perp$ (resp.\ $V(z),V(w)\in \ell_2^\perp$).  Their pre-images in $\tX$ are the reducible fibers of type $\tilde{D}_6$ of $|F_1|$ (resp.\ $|F_2|)$.  The image of the curve $\tX^{\tsigma}$ in $Q$ is a smooth curve $B$ of bidegree $(2,2)$ that passes through the vertices of the quadrangle of lines $L = L_0^{(1)}+L_\infty^{(1)}+L_0^{(2)}+L_\infty^{(2)}$.

We summarize our discussion above with the following assertion. 

\begin{proposition}
The morphism $\tilde{\phi}_\sigma\colon \tX/(\tsigma)\to S(\sigma)$ is the blow-up of the following eight points on $S(\sigma)$: four vertices of the quadrangle of lines $L$ and infinitely near points to them corresponding to the tangent direction of the curve $B$. The pre-images of the two fibers in $|\frakf_i|$ corresponding to the ramification points of the projection map $B\to |\frakf_i|^*\cong \bbP^1$ are reducible fibers of\, $|F_i|$ of type $\tilde{A}_1^*$.
\end{proposition}

So far, we have found that the bitangent surface $\Bit(X)$ contains three irreducible components of bidegree $(m,n) = (1,1)$. They correspond to three bitangent involutions $\sigma_i$. The surface also contains four irreducible components of bidegree $(0,1)$, corresponding to the trope-conics $C_{ijk}$.

\begin{theorem}\label{(3,7)}
Let $X$ be the Kummer quartic surface $X$ in characteristic $2$ associated with an ordinary curve of genus $2$.  Then, the surface $\Bit(X)$ is of bidegree $(m,n) = (3,7)$ in $G_1(\bbP^3)$; it consists of seven irreducible components, three of\, bidegree $(1,1)$, corresponding to three bitangent involutions, and four of\, bidegree $(0,1)$, corresponding to four tropes.
\end{theorem}

\begin{proof} 
For any plane $\Pi$ containing $p$, the plane quartic curve $X\cap \Pi$ has bitangent lines forming a line in $G_1(\bbP^3)$, which implies that the plane quartic curve is a double conic.  This is impossible for a normal quartic~$X$.  As we mentioned before, in Lemma~\ref{lemSkew}, the surface $\Bit(X)$ does not contain $\alpha$-planes.  Thus it is enough to show that the class $n$ of $\Bit(X)$ is equal to $3+4= 7$.  This is the number of bitangent lines contained in a general plane in $\bbP^3$.  We have already shown that $n \geq 7$.  On the other hand, it is known that the number of bitangent lines of a smooth plane quartic in characteristic $2$ is $7,4,2$, or $1$ if the Hasse--Witt invariant is equal to $3,2,1$, or $0$, respectively; see \cite{SV}. The assertion now follows.
\end{proof}

\begin{remark}
Consider the standard inversion transformation
$$
T\colon [x,y,z,w]\longmapsto [yzw, xzw, xyw, xyz]
$$
of $\bbP^3$. The restriction of $T$ to $\tilde{X}$ is a fixed-point-free involution $\sigma$ with $\tilde{X}/(\sigma)$ isomorphic to an Enriques surface.  Let us consider the Pl\"ucker embedding of the corresponding congruence of lines.  The six minors $p_{ij}$ of the matrix
$$\left(
    \begin{array}{cccc}
       x & y & z & w \\
       yzw & xzw & xyw & xyz \\
    \end{array}
  \right)$$
are given by
\begin{multline*}
[p_{12}, p_{13}, p_{14}, p_{23}, p_{24}, p_{34}]\\
=\left[zw\left(x^2+y^2\right),\ yw\left(x^2+z^2\right),\ yz\left(x^2+w^2\right),\ xw\left(y^2+z^2\right),\ xz\left(y^2+w^2\right),\ xy\left(z^2+w^2\right)\right],
\end{multline*}
which satisfies
$$
p_{12}p_{34}+ p_{13}p_{24}+p_{14}p_{23} = xyzw((x+y)(z+w)+(x+z)(y+w)+(x+w)(y+z))^2=0, 
$$
and the cubic equation
\begin{equation}\label{cubic}
p_{12}p_{13}p_{23}+ p_{12}p_{14}p_{24} + p_{13}p_{14}p_{34} + p_{23}p_{24}p_{34}=0
\end{equation}
(see \cite[Equation~(33)]{Em}).  

\begin{lemma}
Let $p_0=[x_0,y_0,z_0,w_0]\in X$ be a general point of the Kummer quartic surface.  The line $\ell$ passing through $p_0$ and $\sigma(p_0)$ is not a bitangent line of\, $X$.
\end{lemma}

\begin{proof}
The following proof was suggested by a referee.  By plugging the parametric equation
$$
[x_0 +ty_0z_0w_0, y_0 + tx_0z_0w_0, z_0+ tx_0y_0w_0, w_0+tx_0y_0z_0]
$$
of the line $\ell$ into the equation \eqref{eq} of $X$, we can see that the term which is linear in $t$ is equal to $(x_0+y_0+z_0+w_0)t$.  Since a general point does not lie in the hyperplane $V(x+y+z+w)$, $\ell$ is not a bitangent line.
\end{proof}

Let $\ell$ be a line passing through $p_0$ and $\sigma(p_0)$.  Let $\ell \cap X = \{p_0, \sigma(p_0), p_0', p_0''\}$.  The image of $\ell$ under $\sigma$ is a cubic curve, and hence, $p_0', p_0''$ are not conjugate to $p_0, \sigma(p_0)$ under the action of $\sigma$.  Thus, the map $X \da G_1(\bbP^3)$ sending $p_0$ to $\ell =\la p_0, \sigma(p_0)\ra$ has degree $2$ onto its image.  Therefore, the Enriques surface $\tilde{X}/(\sigma)$ can be embedded into $G_1(\bbP^3)$ satisfying a cubic relation \eqref{cubic}. This suggests the following question.

\begin{question}
  Is the Enriques surface $\tilde{X}/( \sigma )$ a Reye congruence of bidegree $(7,3)$?
\end{question}
\end{remark}

\section{Kummer quartic surfaces in characteristic 2, 2-rank 1 case}\label{sec5}

In this section, we discuss the case of Kummer quartic surfaces associated with curves of $2$-rank $1$.  The Kummer quartic surface $X_1$ associated with a curve of genus~$2$ and $2$-rank $1$ is given by
\begin{equation}\label{KummerEq1}
X_1= V\left(\beta^2x^4+\alpha^2x^2z^2 + x^2zw + xyz^2 + y^2w^2 + z^4\right),
\end{equation}
where $\alpha, \beta$ are constants with $\beta \ne 0$; see \cite[Section~3]{Du}.  The surface $X_1$ has exactly two singular points
$$
p_1=[0,0,0, 1],\quad p_2=[0, 1, 0, 0]
$$
of type $D_8$, see \cite{Sh,Ka}, and contains two tropes defined by the hyperplane sections $x=0$ and $z=0$, respectively.  The two tropes meet at $p_1$ and $p_2$.

Consider the skew lines $\ell= V(x, y)$ and $\ell'= V(z, w)$.  In the same way as in the ordinary case, we see that the lines meeting $\ell, \ell'$ are bitangent lines of $X_1$.  The corresponding bitangent involution of $X_1$ is given by the restriction of a Cremona involution
\begin{equation}
\sigma_1\colon [x,y,z,w]\longmapsto \left[xz^2, yz^2, \beta x^2z, \beta x^2w\right].
\end{equation}
Thus, $\Bit(X_1)$ contains a smooth quadric surface $S(\sigma_1)$. 

There exists another bitangent involution
\begin{equation}
\sigma_2\colon [x,y,z,w]\longmapsto \left[x^2z, x^2w, xz^2, yz^2\right],
\end{equation}
which is a composite of $\sigma_1$ with a projective linear transformation
$$
\tau \colon [x,y,z,w]\lra [z, w, \beta x,\beta y].
$$
Note that $\tau$ is induced from the translation by the non-zero 2-torsion of the Jacobian of the curve. 

Now, let us consider the Pl\"ucker embeddings of the congruence of bitangent lines defined by $\sigma_2$.  The six minors $p_{ij}$ of the matrix
$$
\left(
    \begin{array}{cccc}
       x & y & z & w \\
       x^2z & x^2w & xz^2 & yz^2 \\
    \end{array}
    \right)
    $$
are given by
\begin{align*}
[p_{12}, p_{13}, p_{14}, p_{23}, p_{24}, p_{34}]
&=\left[x^2(xw+yz),\ 0,\ xz(xw+yz),\ xz(xw+yz),\ (xw+yz)^2,\ z^2(xw+yz)\right]\\
&=\left[x^2, 0, xz, xz, xw+yz, z^2\right].
\end{align*}
They satisfy
$$
p_{13}= 0,\quad p_{14}+p_{23}=0,\quad p_{12}p_{34}+p_{14}p_{23}=0.
$$
Thus, the congruence of lines $S(\sigma_2)$ is a quadric cone, a special linear section of the Grassmannian quadric.

We now conclude as follows. 

\begin{theorem}\label{(2,4)}
  Let $X_1$ be the Kummer quartic surface associated with a smooth curve of genus~$2$ and $2$-rank $1$.  Then, $\Bit(X_1)$ is of bidegree $(2,4)$ in $G_1(\bbP^3)$.  It consists of two tropes and two quadric surfaces; one of the latter is a smooth quadric $S(\sigma_1)$, and the other is a quadric cone $S(\sigma_2)$.
\end{theorem}

\begin{proof}
We know that $\Bit(X)$ contains two $\beta$-planes and two irreducible components of bidegree $(1,1)$.  Taking the partial derivatives, we find that they are linearly independent, and hence, there are no inseparable projection centers. So, $\Bit(X)$ does not contain $\alpha$-planes.  The equation of $X$ is of the form
$$
Q^2+x^2zw+xyz^2 = 0.
$$
Substituting $w = ax+by+cz$, we obtain the equation of a general plane section $q^2+x^2z(ax+by)+xyz^2 = 0$. It is immediate to check that it is projectively equivalent to a plane quartic of type~\eqref{wall-2} from Proposition~\ref{wall}. Thus, the number of
bitangent lines is equal to $4$, and hence, the class of $\Bit(X)$ is equal to $4$.  So, we have found all irreducible components of $\Bit(X)$: two $\beta$-planes and two components of bidegree $(1,1)$.
\end{proof}

\begin{remark}
Let $\ell_0 = V(x, z)$ be the line which is the intersection of the two trope-hyperplanes.  Then, any bitangent line defined by $\sigma_2$ meets $\ell_0$ as follows.  Let $p= [x_0,y_0,z_0,w_0] \in \bbP^3$ be a general point and $q=\sigma_2(p)$.  The bitangent line $\la p, q\ra$ is given by
$$
\left[x_0(s+tx_0z_0),\ sy_0+tx_0^2w_0,\ z_0(s+tx_0z_0),\ sw_0+ty_0z_0^2\right], 
$$
where $[s,t]\in \bbP^1$ is a parameter.  The line $\la p, q\ra$ meets $\ell_0$ at $[0,x_0,0,z_0]$ when $[s, t]= [x_0z_0,1]$.
\end{remark}

\section{Kummer quartic surfaces in characteristic 2, supersingular case}\label{sec6}

In this section, we consider the supersingular case.  In this case, the Kummer quartic surface $X_0$ is given by the equation
\begin{equation}\label{ssKumeq}
X_0= V(x^3w +\alpha x^3y + x^2yz + \alpha^2x^2z^2 + xy^3+ y^2w^2 + z^4),
\end{equation}
where $\alpha$ is a constant; see \cite[Equation~(5.1)]{Duc} and \cite[Section~3]{Du}.  The Kummer surface has one singular point $p_0 = [0,0,0,1]$, which is an elliptic singularity of type $\raise0.2ex\hbox{\textcircled{\scriptsize{4}}}_{0,1}^1$ in the sense of Wagreich, see \cite{Ka}, and contains a trope defined by the hyperplane $x=0$.

Let us consider the following Cremona involution:
\begin{equation}\label{}
T\colon  [x,y,z,w]\longmapsto [x^3, x^2y, \alpha x^3 + x^2z + xy^2, \alpha x^2y + x^2w + y^3].
\end{equation}
The involution $T$ preserves $X_0$ and restricts to a bitangent involution $\sigma$ of $X_0$.  This follows from a direct calculation.

Let us consider the Pl\"ucker embeddings of bitangent lines. The six minors $p_{ij}$ of the matrix
$$
\left(
    \begin{array}{cccc}
       x & y & z & w \\
       x^3 & x^2y & \alpha x^3 + x^2z + xy^2 & \alpha x^2y + x^2w + y^3 \\
    \end{array}
    \right)
    $$
are given by
\begin{multline*}
  [p_{12}, p_{13}, p_{14}, p_{23}, p_{24}, p_{34}]\\
  \begin{aligned}
&=\left[0,\ x^2(\alpha x^2+y^2),\ xy(\alpha x^2+y^2),\ xy(\alpha x^2+y^2),\ y^2(\alpha x^2+y^2),\ (xw+yz)(\alpha x^2+y^2)\right]\\
    & =\left[0, x^2, xy, xy, y^2, xw+yz\right].  
  \end{aligned}
\end{multline*}
They satisfy
$$
p_{12}= 0,\quad p_{14}+p_{23}=0, \quad p_{13}p_{24}+p_{14}p_{23}=0.
$$
Thus, the congruence of lines $S(\sigma)$ is a quadric cone.

\begin{theorem}\label{(1,2)}
Let $X_0$ be the Kummer quartic surface associated with a supersingular curve.  Then, $\Bit(X_0)$ is of bidegree $(1,2)$ in $G_1(\bbP^3)$, and it consists of a trope and a quadric cone.
\end{theorem}

\begin{proof}
Taking the partial derivatives, we find again that there are no inseparable projection centers. So, $\Bit(X)$ does not contain $\alpha$-planes.  Since we have already found irreducible components of $\Bit(X)$ of bidegrees $(0,1)$ and $(1,1)$, it is enough to check that the number of bitangent lines of a general plane section is equal to $2$.  This can be seen directly by plugging $w = ax+by+cz$ into the equation of $X$ and reducing the non-square part of the obtained equation to the form~\eqref{wall-3} from Proposition~\ref{wall}.
\end{proof}

\begin{remark}
Let $\ell_0=V(x, y)$.  Then, any bitangent line meets $\ell_0$ as follows.  Let $p= [x_0,y_0,z_0,w_0] \in \bbP^3$ be a general point and $q=\sigma(p)$.  The bitangent line $\la p, q\ra$ is given by
$$
[x_0(s+tx_0^2), y_0(s+tx_0^2), z_0(s+tx_0^2) + tx_0(\alpha x_0^2 + y_0^2), w_0(s+tx_0^2)+ty_0(\alpha x_0^2+y_0^2)],
$$
where $[s,t]\in \bbP^1$ is a parameter.  The line $\la p, q\ra$ meets $\ell_0$ at $[0,0,x_0,y_0]$ when $[s, t]= [x_0^2, 1]$.
\end{remark}


\end{document}